\documentclass[11pt]{amsart}

\newtheorem{thm}{Theorem}[section]
\newtheorem{cor}[thm]{Corollary}
\newtheorem{lem}[thm]{Lemma}
\newtheorem{prop}[thm]{Proposition}
\newtheorem{claim}[thm]{Claim}
\theoremstyle{definition}

\newtheorem{example}[thm]{Example}
\theoremstyle{remark}
\newtheorem{rem}[thm]{Remark}
\newtheorem{question}[thm]{Question}
\numberwithin{equation}{section}
\usepackage{amsmath,amssymb,graphicx,bbm,amsthm,hyperref,mathrsfs,tikz,palatino,changepage,enumerate,color,ifpdf,booktabs,multirow,tabularx,appendix,makecell}
\usepackage[all]{xy}
\usetikzlibrary{shapes.geometric,arrows}

\newcommand{\Arf}{{\rm Arf}}
\newcommand{\Aut}{{\rm Aut}}
\newcommand{\MCG}{{\rm MCG}}

\begin{document}

\title[]{Extendability over the $4$-sphere and invariant spin structures of surface automorphisms}

\author{Weibiao Wang}
\address{Morningside Center of Mathematics, Academy of Mathematics and Systems Science, Chinese Academy of Sciences, Beijing, 100190, China}
\email{wangweibiao@amss.ac.cn}

\author{Zhongzi Wang}
\address{School of Mathematical Sciences, Peking University, Beijing, 100871, China}
\email{wangzz22@stu.pku.edu.cn}



\begin{abstract} 
It is known that an automorphism of $F_g$, the oriented closed surface of genus $g$, 
is extendable over the 4-sphere $S^4$ if and only if it has a bounding invariant spin structure \cite{WsWz}. 

We show that each automorphism of $F_g$ has an invariant spin structure, and obtain  a stably extendable result: Each automorphism of $F_g$ is extendable over $S^4$ after a connected sum with the identity map on the torus. 
Then each automorphism of an oriented once punctured surface is extendable over $S^4$.

For each $g\neq 4$, we construct a periodic map  on $F_g$ that is not extendable over $S^4$, and we prove that every periodic map on $F_4$ is extendable over $S^4$,
which answer a question in \cite{WsWz}. 

We illustrate for an automorphism $f$ of $F_g$, how to find its invariant spin structures, bounding or not; and once $f$ has   a bounding invariant spin structure,
how to construct an embedding $F_g\hookrightarrow S^4$ so that $f$ is extendable with respect to this embedding.
\end{abstract}

\date{}
\maketitle
\tableofcontents

\section{Introduction}\label{sect:introduction}

We will work in the category of oriented smooth manifolds. 
For an oriented compact manifold $M$, we use $\Aut(M)$ to denote the group of its automorphisms, i.e., orientation-preserving self-diffeomorphisms. 
Let $F_g$ be the oriented closed surface of genus $g\geq 1$. 
And let $\MCG(F_g)$ be the mapping
class group of $F_g$, i.e., the group of
smooth isotopy classes of automorphisms of $F_g$. 
For $f\in\Aut(F_g)$, denote its mapping class by $[f]$.

An element $f$ in $\Aut(F_g)$ is \textit{extendable over $M$ with respect to an embedding $e:F_g\hookrightarrow M$} if there exists an element $\tilde f$ in $\Aut(M)$ such that $$\tilde f\circ e=e\circ f.$$
Call a surface map $f\in\Aut(F_g)$ is \textit{extendable over $M$}, if $f$ is extendable over $M$ with respect to some embedding $e$.
Call $[f]$ is extendable if $f$ is extendable.
For an automorphism of a general manifold $N$, we can similarly talk about its extendability over another manifold $M$.

It is known that each $f\in\Aut(F_g)$ is extendable over $S^5$, thus is extendable over any $n$-dimensional manifold with $n\geq 5$
\cite{Hir}.
On the other hand, the codimension-$1$ extendability is an active topic, see \cite{NWW} and the references therein. 
We will focus on the codimension-$2$ case.

An embedding $e:F_g\hookrightarrow S^4$ is said to be \textit{trivial} if $e(F_g)$ bounds an embedded handlebody.
For a trivial embedding of the torus $F_1$, Montesinos showed in 1983 that $f\in\Aut(F_1)$ is extendable over $S^4$ with respect to it if and only if $f$ preserves the induced Rohlin form \cite{Mon}.  
In 2002 the same criterion for $F_g$ was proved by Hirose in a comprehensive work \cite{Hi1}. 
In 2012, it was proved that if $N$ is an $n$-dimensional manifold and $f\in\Aut(N)$ is extendable  over $S^{n+2}$ with respect to $e: N\hookrightarrow S^{n+2}$, then $f$ preserves  the spin structure induced by $e$ \cite{DLWY}. 
In the case of $F_g\hookrightarrow S^4$, the induced spin structure and the induced Rohlin form coincide \cite[Lemma 3.2]{WsWz}.
Based on \cite{Hi1} and \cite{DLWY}, 
as well as the fact that $\Aut(F_g)$ acts transitively on the set of bounding spin structures on $F_g$ \cite[\S 5]{Jo} (also in \cite[Lemma 4.2]{DLWY}), it was proved in 2021 that 

\begin{thm}\label{WW3.1} \cite[Theorem 3.1]{WsWz}
An automorphism $f$ of $F_g$ is extendable over $S^4$ if and only if $f$ has an invariant bounding spin structure on $F_g$.
\end{thm}

Theorem \ref{WW3.1} transfers a problem in geometric flavor to a problem in more algebraic flavor (for details see Theorem \ref{thm:extendability}), in particular, when $f$ is given by a product of Dehn twists along standard generators of $H_1(F_g; \mathbb{Z}_2)$
(for details see Section \ref{sect:low-genus}). 

It is natural to ask whether every $f\in\Aut(F_g)$ has an invariant spin structure. 
The answer is YES. 

\begin{thm}\label{IS}
Every automorphism of $F_g$ has an invariant  spin structure.
\end{thm}

This was first proved by Atiyah \cite[Proposition 5.2]{Ati}.
We will reprove it using quadratic forms, indeed get a stronger result: The number of $f$-invariant spin structures equals $2^d$, where
$d$ is the dimension of the kernel of $f_*-id: H_1(F_g;\mathbb{Z}_2)\to H_1(F_g;\mathbb{Z}_2)$ (Theorem  \ref{SIS}). Moreover, to check whether there is a bounding one among those $2^d$ elements, one needs only to check $\frac{d^2+d+2}2$ elements among them (Proposition \ref{prop:Arf}). 
	
Based on Theorem  \ref{WW3.1} and Theorem \ref{IS}, we have the following stably extendable result.
(The connected sum of two mapping classes will be defined in Section \ref{sect:applications}.)

\begin{thm}\label{stable}
	Let $I_T$ be the identity on the torus $T$. 
	For any $[f]\in  \MCG(F_g)$, the connected sum $[f]\# [I_T]\in \MCG(F_{g+1})$ is extendable over $S^4$.
\end{thm}

Let $F_{g,k}$ be the compact oriented surface of genus $g$ and with $k$ boundary components. 

\begin{thm}\label{puncture}
 Every automorphism of $F_{g,k}$ is extendable over $S^4$ when $k$ is odd.
\end{thm}

It is an application of Theorem \ref{WW3.1} that for $g\geq 1$, a Wiman map on $F_g$ (i.e., a periodic map of the maximum order $4g+2$) is extendable over $S^4$ if and only if  $g\equiv 0\text{ or }3\,({\rm mod}\,4)$ \cite[Theorem 1.2 (1)]{WsWz}. 
Note that for each $g\ge 1$, there exists a mapping class on $F_g$ that is not extendable over $S^4$ \cite[Corollary 1.4]{DLWY}, but the explicit map was unknown before for $g>1$. 
The following question is raised in \cite[Remark 1 (4)]{WsWz}.

\begin{question}
	For each $g\geq 1$, is there a periodic automorphism of $F_g$ that is not extendable over $S^4$?
\end{question}

We give a complete answer:

\begin{thm}\label{non-extendable}
	Suppose that $g\geq 1$ and $g\neq 4$, then there exists a periodic map  on $F_g$ that is not extendable over $S^4$.
	Indeed the map can be of period $8$ except when $g=1$.
	In contrast, every periodic map on $F_4$ is extendable over $S^4$.
\end{thm}

This paper is organized as follows.
In Section \ref{sect:quadratic} we introduce the basic tools, spin structures and quadratic forms on surfaces.
In Section \ref{sect:IS} we prove Theorem \ref{IS} and get more information about the invariant spin structures.
In Section \ref{sect:applications} we introduce a method to constructing surface maps on connected sums, then prove Theorem \ref{stable} and Theorem \ref{puncture}. 
We also provide explicit examples, from which the first half of Theorem \ref{non-extendable} follows.
The detailed calculations needed in the proof of Theorem \ref{non-extendable} are left to Section \ref{sect:low-genus},
which also provides a general way to find invariant spin structures for an automorphism $f$ of $F_g$, bounding or not, especially when $f$ is given by a product of Dehn twists along standard generators of $H_1(F_g; \mathbb{Z}_2)$.
In Section \ref{sect:embedding}, we illustrate once $f$ has a  bounding invariant spin structure
$\sigma$ (i.e., $f$ is extendable over $S^4$),
how to construct an embedding $e: F_g\hookrightarrow S^4$ so that the induced spin structure is $\sigma$
(i.e., $f$ is extendable with respect to  $e$).

\begin{rem} 
	Let $\mathcal{T}_g$ be the finite index subgroup of  $\MCG(F_g)$ which acts trivially on $H_1(F_g;\mathbb{Z}_2)$. 
	Then for any $f\in \Aut(F_g)$ and $[h]\in \mathcal{T}_g$, $h \circ f$ is extendable over $S^4$ if and only if $f$ is extendable over $S^4$. 
	Each $[h]\in \mathcal{T}_g$ is extendable over $S^4$ with respect to any trivial embedding $F_g\hookrightarrow S^4$.
\end{rem}

\section{Spin structures and quadratic forms on surfaces}\label{sect:quadratic}

Spin structure on surfaces is a subject which has many applications as well as  has many faces.
One may refer to \cite[Chapter IV and Appendix]{Ki} for an introduction.

Take the first homology group $H_1(F_g;\mathbb{Z}_2)$ of $F_g$ with coefficients in $\mathbb{Z}_2$. 
For $x,y\in H_1(F_g;\mathbb{Z}_2)$, denote by $x\cdot y$ their mod-2 intersection number.
A \textit{quadratic form} on $H_1(F_g;\mathbb{Z}_2)$ is a function $q:H_1(F_g;\mathbb{Z}_2)\to\mathbb{Z}_2$ satisfying
\begin{displaymath}
	q(x+y)=q(x)+q(y)+x\cdot y,\, \forall\,x,y\in H_1(F_g;\mathbb{Z}_2).
\end{displaymath}
For convenience, we use $\mathcal{Q}(F_g)$ to denote the set of quadratic forms on $H_1(F_g;\mathbb{Z}_2)$. 
An element $q\in\mathcal{Q}(F_g)$ is determined by its values on a basis of $H_1(F_g;\mathbb{Z}_2)$, so $\mathcal{Q}(F_g)$ is isomorphic to $H^1(F_g;\mathbb{Z}_2)$ as a set.

\input{symplectic.TpX}

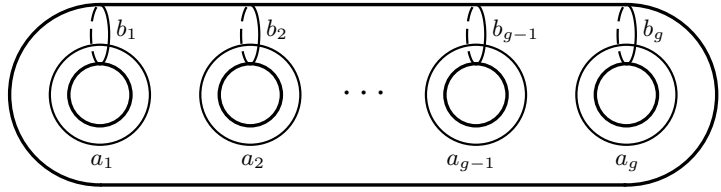
Suppose that $\{a_1,b_1,\cdots,a_g,b_g\}$ is a \textit{symplectic basis} of $H_1(F_g;\mathbb{Z}_2)$ with respect to the mod-$2$ intersection form, i.e., $a_i\cdot b_j=1$ if and only if $i=j$, and $a_i\cdot a_j=b_i\cdot b_j=0$ for any $i,j$. 
See Fig. \ref{fig:symplectic} for a typical example.
The \textit{Arf invariant} of a quadratic form $q$ on $H_1(F_g;\mathbb{Z}_2)$ is defined by
\begin{displaymath}
	\Arf(q)=\sum_{i=1}^{g}q(a_i)q(b_i)\in\mathbb{Z}_2,
\end{displaymath}
which is independent of the choice of the symplectic basis.

Spin structures on $F_g$ are in 1-1 correspondance with quadratic forms on $H_1(F_g;\mathbb{Z}_2)$, and bounding spin structures correspond exactly to 
those quadratic forms with Arf invariant $0$ (cf. \cite[p. 36]{Ki}).

For $f\in\Aut(F_g)$ and $q\in\mathcal{Q}(F_g)$, naturally we have the pull-back quadratic form $f^*q\in\mathcal{Q}(F_g)$: 
\begin{displaymath}
	f^*q(x)=q(f_*(x)),\,\forall\,x\in H_1(F_g;\mathbb{Z}_2).
\end{displaymath}

Now we restate  Theorem \ref{WW3.1}.

\begin{thm}\cite[Theorem 3.1]{WsWz}\label{thm:extendability}
	For any $f\in\Aut(F_g)$, $f$ is extendable over $S^4$ if and only if there exists a quadratic form $q$ on $H_1(F_g;\mathbb{Z}_2)$ such that $\Arf(q)=0$ and $f^*q=q$.
\end{thm}

According to the above theorem, the extendability over $S^4$ of $f\in\Aut(F_g)$ is totally determined by the induced isomorphism $f_*$ on $H_1(F_g;\mathbb{Z}_2)$.
Indeed once $f_*$ is given, there is an algorithm to determine whether $f$ is extendable over $S^4$, see Section \ref{sect:low-genus}.

\section{Every surface automorphism has an invariant spin structure}\label{sect:IS}

We will prove the following  stronger version of Theorem \ref{IS}.

\begin{thm}\label{SIS}
	For every $f\in \Aut(F_g)$ with $f_*: H_1(F_g;\mathbb{Z}_2)\rightarrow H_1(F_g;\mathbb{Z}_2)$, 
	
	{\rm (1)} there exists a quadratic form $q$ on $H_1(F_g;\mathbb{Z}_2)$ such that $f^*q=q$;

	{\rm (2)} the number of $f$-invariant spin structures equals 
	$$|ker(f_*-id)|=2^{{\rm dim}ker(f_*-id)}.$$
\end{thm}

Logically (1) follows from (2). 
But to prove (2), we need first prove (1).

For any $\mathbb{Z}_2$-linear map $\xi:H_1(F_g;\mathbb{Z}_2)\rightarrow \mathbb{Z}_2$ and any quadratic forms $q,q'$ on $H_1(F_g;\mathbb{Z}_2)$, we denote $\xi+q$ the function defined by $x\mapsto \xi(x)+q(x)$ for all $x\in H_1(F_g;\mathbb{Z}_2)$, and denote $q'-q$ the map defined by $x\mapsto q'(x)-q(x)$ for all $x\in H_1(F_g;\mathbb{Z}_2)$. 

\begin{lem}\cite[Lemma 2]{Jo}\label{lem:linear}
	With the above assumption, $\xi+q$ is a quadratic form on $H_1(F_g;\mathbb{Z}_2)$ and $q'-q$ is $\mathbb{Z}_2$-linear. 
\end{lem}

We also need the following known fact from algebra.

\begin{lem}\label{lem:factor}
	Suppose $V,W$ are $\mathbb{Z}_2$-vector spaces, and $\psi:V\rightarrow \mathbb{Z}_2$ and $\phi:V\rightarrow W$ are $\mathbb{Z}_2$-linear. 
	Then there exists a $\mathbb{Z}_2$-linear map $\xi:W\rightarrow \mathbb{Z}_2$ such that $\xi\circ\phi=\psi$ if and only if $ker\phi\subseteq ker\psi$.
\end{lem}

\begin{proof}[Proof of Theorem \ref{SIS}] 
	(1) Pick an arbitrary quadratic form $q_0\in \mathcal{Q}(F_g)$ and fix it. 
	For $q\in\mathcal{Q}(F_g)$, let $\xi=q-q_0$ be the $\mathbb{Z}_2$-linear map such that $q=\xi+q_0$. 
	Transform the equation $f^*q=q$:
	\begin{displaymath}
		\begin{split}
			f^*q=q & \Leftrightarrow f^*(\xi+q_0)=\xi+q_0\\
			&\Leftrightarrow f^*\xi+f^*q_0=\xi+q_0\\
			& \Leftrightarrow f^*\xi-\xi=q_0-f^*q_0\\
			& \Leftrightarrow \xi\circ(f_*-id)=q_0-f^*q_0.
		\end{split}
	\end{displaymath}
	By Lemma \ref{lem:linear}, $q_0-f^*q_0$ is a $\mathbb{Z}_2$-linear map. 
	Then by Lemma \ref{lem:factor}, we only need to verify $ker(f_*-id)\subseteq ker(q_0-f^*q_0)$. 
	It is straightforward: 
	for any $x\in ker(f_*-id)$, $f_*(x)=x$, so
	\begin{displaymath}
		(q_0-f^*q_0)(x)=q_0(x)-q_0(f_*(x))=q_0(x)-q_0(x)=0.
	\end{displaymath}
	
	(2) By (1) we may fix $q_0\in\mathcal{Q}(F_g)$ with $f^*q_0=q_0$.
	Then
	\begin{displaymath}
		\begin{split}
			f^*q=q & \Leftrightarrow \xi\circ(f_*-id)=q_0-f^*q_0 \\
			& \Leftrightarrow \xi|_{Im(f_*-id)}=0.
		\end{split}
	\end{displaymath}
	Let $W$ be a direct sum complement of the image $Im(f_*-id)$ in $H_1(F_g;\mathbb{Z}_2)$. Then the function
	$$q\mapsto \xi|_W$$
	defines a bijection between 
	the set of $f$-invariant quadratic forms to the hom-space ${\rm Hom}(W,\mathbb{Z}_2)$.
	So the number of $f$-invariant quadratic forms equals the cardinality of ${\rm Hom}(W,\mathbb{Z}_2)$. 
	Note that
	$${\rm Hom}(W,\mathbb{Z}_2)\cong W$$
	as $\mathbb{Z}_2$-vector spaces, so we have 
	\begin{displaymath}
		\begin{split}
			{\rm dim}\,{\rm Hom}(W,\mathbb{Z}_2)={\rm dim}W &={\rm dim}H_1(F_g;\mathbb{Z}_2)-{\rm dim}Im(f_*-id)\\
			&={\rm dim}ker(f_*-id).
		\end{split}
	\end{displaymath}
	Thus $|{\rm Hom}(W, \mathbb Z_2)|=|ker(f_*-id)|$, and the theorem follows.
\end{proof}

\begin{prop}\label{prop:Arf}
	Suppose that $f$ is an automorphism of $F_g$ and $q_0$ is an $f$-invariant quadratic form on $H_1(F_g;\mathbb{Z}_2)$.
	Let $\{\xi_1,\xi_2,\cdots,\xi_d\}$ be a basis of the $\mathbb{Z}_2$-linear space 
	$$\{\xi:H_1(F_g;\mathbb{Z}_2)\to\mathbb{Z}_2\,|\,\xi\circ(f_*-id)=0\}.$$
	Then there exists an $f$-invariant quadratic form with Arf invariant $0$ if and only if $0$ belongs to the set
	$$\{\Arf(q_0),\Arf(q_0+\xi_i),\Arf(q_0+\xi_i+\xi_j)\,|\,i,j=1,2,\cdots,d,i\neq j\}.$$
\end{prop}

\begin{proof}
	By the proof of Theorem \ref{SIS} (2), $q\in\mathcal{Q}(F_g)$ is $f$-invariant if and only if $q-q_0$ belongs to
	$$\{\xi:H_1(F_g;\mathbb{Z}_2)\to\mathbb{Z}_2\,|\,\xi\circ(f_*-id)=0\},$$
	i.e., $q$ has the form 
	$$q_0+\sum_{i=1}^{d}t_i\xi_i,\,t_1,t_2,\cdots,t_d\in\mathbb{Z}_2.$$
	So it suffices to prove the ``only if" part.
	Take a symplectic basis $\{a_1,b_1,\cdots,a_g,b_g\}$ of $H_1(F_g;\mathbb{Z}_2)$.
	Expanding
	\begin{displaymath}
		\begin{split}
			\Arf\left(q_0+\sum_{i=1}^{d}t_i\xi_i\right) &  =\sum_{j=1}^{g}\left(q_0+\sum_{i=1}^{d}t_i\xi_i\right)(a_j)\cdot\left(q_0+\sum_{i=1}^{d}t_i\xi_i\right)(b_j)\\
			&=\sum_{j=1}^{g}\left(q_0(a_j)+\sum_{i=1}^{d}t_i\xi_i(a_j)\right)\left(q_0(b_j)+\sum_{i=1}^{d}t_i\xi_i(b_j)\right),
		\end{split}
	\end{displaymath}
	we see that it can be expressed as a polynomial of $t_1,t_2,\cdots,t_d$:
	$$\Arf\left(q_0+\sum_{i=1}^{d}t_i\xi_i\right) =A+\sum_{i=1}^{d}B_it_i+\sum_{1\leq i<j\leq d}C_{ij}t_it_j.$$
	The coefficients $A,B_i,C_{ij}$ can be determined as follows:
	\begin{itemize}
		\item Take $t_i=0$, $\forall i$, then we see $A=\Arf(q_0);$
		\item Fix $i$ and take $t_i=1,t_j=0,\forall j\neq i$, then we see $B_i=\Arf(q_0+\xi_i)-A$;
		\item Fix $i<j$ and take $t_i=t_j=1,t_k=0,\forall k\neq i,j$, then we see
		$$C_{ij}=\Arf(q_0+\xi_i+\xi_j)-A-B_i-B_j.$$
	\end{itemize}
	If $0$ does not belong to the set
	$$\{\Arf(q_0),\Arf(q_0+\xi_i),\Arf(q_0+\xi_i+\xi_j)\,|\,i,j=1,2,\cdots,d,i\neq j\},$$
	then $A=1,B_i=C_{ij}=0$, which makes the Arf invariant of any $f$-invariant quadratic form $q_0+\sum_{i=1}^{d}t_i\xi_i$ a constant $1$. 
	This completes our proof.
\end{proof}

\begin{cor}\label{cor:unique}
	Suppose that $f$ is an automorphism of $F_g$ and $f_*-id:H_1(F_g;\mathbb{Z}_2)\rightarrow H_1(F_g;\mathbb{Z}_2)$ is bijective. 
	Then there exists a unique $f$-invariant quadratic form $q$ on $H_1(F_g;\mathbb{Z}_2)$. 
	Moreover, $q$ is explicitly given by 
	$$q(x)=x\cdot (f_*-id)^{-1}(x),\,\forall\,x\in H_1(F_g;\mathbb{Z}_2).$$
\end{cor}

\begin{proof} 
	The existence and uniqueness follow immediately from Theorem \ref{SIS}, so we only need to prove the moreover part. 
	For any $x\in H_1(F_g;\mathbb{Z}_2)$, as $f_*-id$ is bijective, we can take $y$ with $x=(f_*-id)(y)=f_*(y)-y.$
	For the $f$-invariant quadratic form $q$, we have $q(f_*(y))=q(y)=-q(y)$, so
	\begin{displaymath}
		\begin{split}
			q(x)=q(f_*(y)-y)=q(f_*(y))+q(y)+f_*(y)\cdot y=f_*(y)\cdot y \\
			=f_*(y)\cdot y-y\cdot y=(f_*(y)-y)\cdot y=x\cdot (f_*-id)^{-1}(x).
		\end{split}
	\end{displaymath} 
\end{proof}

\section{Applications on extending surface automorphisms over $S^4$}\label{sect:applications}

For any two mapping classes $[f_1]\in\MCG(F_{g_1}), [f_2]\in\MCG(F_{g_2})$, we can always choose representatives $f_1$ and $f_2$ such that each $f_i(i=1,2)$ coincides with the identity map on a closed disk $D_i$ in $F_{g_i}$. 
When we glue $F_{g_1}-{\rm Int}(D_1)$ with $F_{g_2}-{\rm Int}(D_2)$ along their oriented boundaries via an orientation-reversing homeomorphism, $f_1|_{F_{g_1}-{\rm Int}(D_1)}$ and  $f_2|_{F_{g_2}-{\rm Int}(D_2)}$ together provide an automorphism $f$ of $F_{g_1+g_2}=F_{g_1}\# F_{g_2}$. 
We call $[f]$ a \textit{connected sum} of $[f_1]$ and $[f_2]$, denoted as $[f]=[f_1]\#[f_2]$.
Similarly we can construct a connected sum of finitely many mapping classes on surfaces. 

\begin{prop}\label{prop:connectedsum}
	Suppose that $g=g_1+g_2+\cdots +g_k(k>0)$, and $[f]\in\MCG(F_g)$ is a connected sum of $[f_i]\in\MCG(F_{g_i})(i=1,2,\cdots,k)$. 
	
	{\rm (1)} If some $f_i$ has both bounding and unbounding invariant spin structures, then $f$ is extendable over $S^4$. 
	
	{\rm (2)} Otherwise $f$ is extendable if and only if the number of those $f_i$ with only unbounding invariant spin structures is even.	
\end{prop}

\begin{proof} 
	(1) By Theorem \ref{IS}, each $f_i$ has an invariant  quadratic form $q_i$ on $H_1(F_{g_i};\mathbb{Z}_2)$.
	They together give an $f$-invariant quadratic form $q$ on
	\begin{displaymath}
		H_1(F_g;\mathbb{Z}_2)\cong\bigoplus_{i=1}^{k}H_1(F_{g_i};\mathbb{Z}_2),
	\end{displaymath}
	which equals $q_i$ on $H_1(F_{g_i};\mathbb{Z}_2)$.
	Without loss of generality, assume that  $f_1$ has both bounding and unbounding invariant spin structures, then we can choose an $f_1$-invariant quadratic form $q_1$ such that $\Arf(q_1)=\sum\limits_{i=2}^{k}\Arf(q_i)$,  thus 
	\begin{displaymath}
		\Arf(q)=\sum_{i=1}^{k}\Arf(q_i)=0.
	\end{displaymath}
	By Theorem \ref{thm:extendability}, $f$ is extendable over $S^4$.
	
	(2) Otherwise let $j$ be the number of $f_i$ with only unbounding invariant spin structures,
	then the number of $f_i$ with only bounding invariant spin structures is $k-j$. 
	Note that $f$ acts on 
	\begin{displaymath}
		F_g=F_{g_1}\# F_{g_2}\#\cdots\# F_{g_k}
	\end{displaymath}
	as $f_i$ on each $F_{g_i}$ part, 
	so any $f$-invariant quadratic form $q$ induces an $f_i$-invariant quadratic form $q_i$ on $H_1(F_{g_i};\mathbb{Z}_2)$.
	It follows that 
	\begin{displaymath}
		\Arf(q)=\sum_{i=1}^{k}\Arf(q_i)=j\times 1 + (k-j)\times 0=j.
	\end{displaymath}
	By Theorem \ref{thm:extendability}, $f$ is extendable over $S^4$ if and only if $j$ is even.
\end{proof}

\begin{proof}[Proof of Theorem \ref{stable}]
	According to \cite[Theorem 3']{Ati}, the number of bounding spin structures on $F_g$ is $2^{g-1}(2^g+1)$, while the number of unbounding spin structures is $2^{g-1}(2^g-1)$.
	Particularly, there are three bounding spin structures and one unbounding spin structure on the torus $T$.
	Clearly the identity map $I_T$ keeps each spin structure on $T$ invariant, so it has both bounding and unbounding invariant spin structures. 
	Then the conclusion follows from Proposition \ref{prop:connectedsum} (1).
\end{proof}

Theorem \ref{puncture} can be stated in a stronger version as follows. 

\begin{prop}\label{prop:finitetype}
	Suppose $f$ is an automorphism of $F_{g, k}$ and 
	there exists a boundary component  whose $f$-orbit is of odd length, then $f$ is extendable over $S^4$.
\end{prop}

\begin{proof}
	Suppose the $k$ boundary components of $F_{g,k}$ has $m$ $f$-orbits, denoted by 
	$$\{b_{i,j}:j=1,2,\cdots,k_i\},\,i=1,2,\cdots,m$$ with $b_{i,j+1}=f^j(b_{i,1})(1\leq j<k_i)$ and $f^{k_i}(b_{i,1})=b_{i,1}$.
	By isotopy, we may assume that the $f$-action on the union of any orbit
	$$\bigcup_{j=1}^{k_i}b_{i,j}\,(i=1,2,\cdots,m)$$
	is periodic, i.e., $f^{k_i}$ coincides with the identity on it.
	
	Take $k$ copies $X_{i,j}(1\leq i\leq m,1\leq j\leq k_i)$ of $F_{1,1}$.
	Glue them with corresponding $b_{i,j}$ to obtain a closed surface $F_{g+k}$. 
	The gluing can be made $f$-equivariantly, so $f$ extends to an automorphism $\hat{f}$ of $F_{g+k}$ which acts on each
	$$\bigcup_{j=1}^{k_i}X_{i,j}\,(i=1,2,\cdots,m)$$
	periodically.
	For any $\hat{f}$-invariant quadratic form $\hat{q}$ on
	$$H_1(F_{g+k};\mathbb{Z}_2)\cong H_1(F_{g,k};\mathbb{Z}_2)\oplus\left(\bigoplus_{i=1}^{m}\bigoplus_{j=1}^{k_i}H_1(X_{i,j};\mathbb{Z}_2)\right),$$
	$\hat{q}$ induces invariant quadratic forms $q$ on $H_1(F_{g,k};\mathbb{Z}_2)$ and $q_i$ on 
	$\bigoplus\limits_{j=1}^{k_i}H_1(X_{i,j};\mathbb{Z}_2)$.
	
	Assume $k_1$ is odd. 
	Fix invariant quadratic forms $q$ and $q_2,q_3,\cdots,q_m$.
	Then choose a quadratic form $q_{1,1}$ on $H_1(X_{1,1};\mathbb{Z}_2)$ such that
	$$\Arf(q_{1,1})=\Arf(q)+\sum_{i=2}^{m}\Arf(q_i).$$
	Let $q_{1,j}(j=1,\cdots,k_1)$ be the pull-back $(\hat{f}^{1-j})^*q_{1,1}$ on $H_1(X_{i,j};\mathbb{Z}_2)$. 
	Then $q_{1,1},q_{1,2},\cdots,q_{1,k_1}$ and $q,q_2,q_3,\cdots,q_m$
	together provide an $\hat{f}$-invariant quadratic form $\hat{q}'$ on $H_1(F_{g+k};\mathbb{Z}_2)$.
	As $\Arf(q_{1,j})=\Arf(q_{1,1})$, we have
	\begin{displaymath}
		\begin{split}
			\Arf(\hat{q}')& =\sum_{j=1}^{k_1}\Arf(q_{1,j})+\Arf(q)+\sum_{i=2}^{m}\Arf(q_i)\\
			& =k_1\times \Arf(q_{1,1})+\Arf(q)+\sum_{i=2}^{m}\Arf(q_i)=0.
		\end{split}
	\end{displaymath}
	By Theorem \ref{thm:extendability}, $\hat{f}$ is extendable over $S^4$, thus so is $f$.
\end{proof}

Below we focus on periodic automorphisms of surfaces.

Suppose that $f_1\in\Aut(F_{g_1}),f_2\in\Aut(F_{g_2})$ are periodic maps of the same order $n$ and each has an isolated fixed point. 
For each $i=1,2$, let $x_i\in F_{g_i}$ be an isolated fixed point of $f_i$, and take an $f_i$-invariant closed neighborhood $U_i\subset F_{g_i}$ with $\varphi_i:U_i\xrightarrow{\cong}D^2$, such that $\varphi_i\circ f_i\circ\varphi_i^{-1}$ is an order-$n$ rotation of the standard closed unit disk $D^2$ (Fig. \ref{fig:periodic-connected-sum}). 
By replacing $f_2$ with some power of it, we may assume that $\varphi_1\circ f_1\circ\varphi_1^{-1}=\varphi_2\circ f_2\circ\varphi_2^{-1}$. 
Glue $F_{g_1}-{\rm Int}(U_1)$ with $F_{g_2}-{\rm Int}(U_2)$ by identifying every point $x\in\partial U_1$ with $\varphi_2^{-1}\circ\varphi_1(x)\in\partial U_2$. 
We obtain $F_{g_1+g_2}\cong F_{g_1}\#\overline{F_{g_2}}$, where $\overline{F_{g_2}}$ is $F_2$ with the orientation reversed.
And there is a map $f$ of period $n$ on $F_{g+n}$, which coincides with $f_i$ on $F_{g_i}-{\rm Int}(U_i)\,(i=1,2)$. 
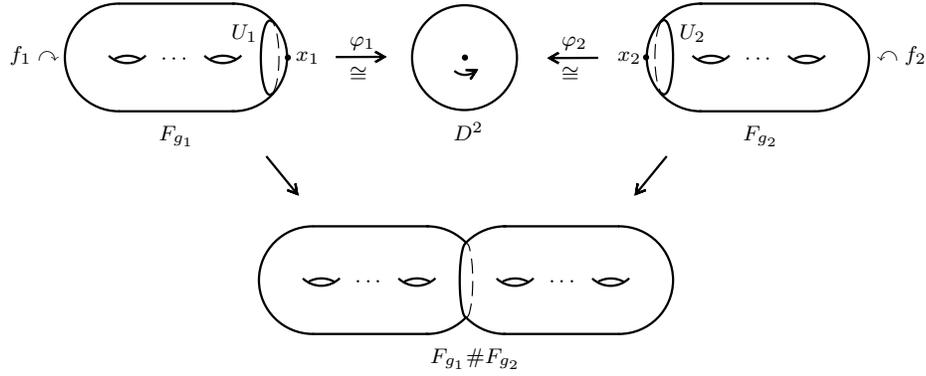
We call $f$ a \textit{periodic connected sum} of $f_1,f_2$. 
Similarly, we can construct a periodic connected sum of several periodic maps if they are of the same period and have enough fixed points for connecting. 

\input{periodic-connected-sum.TpX}

With the same argument for Proposition \ref{prop:connectedsum} we have the following conclusion.

\begin{prop}\label{prop:periodicconnectedsum}
	Suppose that $g=g_1+g_2+\cdots g_k(k>0)$, and $f\in\Aut(F_g)$ is a periodic connected sum of $f_i\in\Aut(F_{g_i})(i=1,2,\cdots,k)$. 
	
	{\rm (1)} If some $f_i$ has both bounding and unbounding invariant spin structures, then $f$ is extendable over $S^4$. 
	
	{\rm (2)} Otherwise $f$ is extendable if and only if the number of those $f_i$ with only unbounding invariant spin structures is even.	
\end{prop}
	
Now we are to prove the first half of Theorem \ref{non-extendable}. 
According to \cite[Theorem 1.2 (1)]{WsWz}, a Wiman map $w_1$ on $F_1$ of period $6$ is an example that is not extendable.
So we only need to prove the following proposition.

\begin{prop}\label{prop:period8}
	Suppose that $g\geq 2$ and $g\neq 4$, then there exists a map of period $8$ on $F_g$ that is not extendable over $S^4$.
\end{prop}

 The proof of Proposiotion \ref{prop:period8} relies on Proposition  \ref{prop:periodicconnectedsum} (2), as well as the contructions based on Example \ref{ex:f233} and Lemma \ref{lem:f233} below.

\begin{example}\label{ex:f233} 		
(1) The surface $F_2$ can be obtained by gluing a regular octagon with each pair of opposite edges identified as in Fig. \ref{fig:f_22}.  
The $\frac{\pi}{4}$-rotation of the octagon  induces a map of period $8$ on $F_2$, denoted by $f_2$.

(2) The surface $F_3$ can be obtained by gluing the edges of one regular octagon and two squares according to the labels shown in Fig. \ref{fig:f_37}.
The $\frac{\pi}{4}$-rotation of the octagon induces a periodic map on the union of the squares, and then induces a map of period $8$ on $F_3$, denoted by $f_3$.

(3)  The surface $F_3$ can also be obtained as in Fig. \ref{fig:f_33}.
Simialrly, the $\frac{\pi}{4}$-rotation of the octagon  induces a map of period $8$ on $F_3$, denoted by $f_3'$. 

Each map above has  two isolated fixed points, one corresponding to the center of  the octagon, and the other corresponding to the vertices.

\input{f_22.TpX}
	
\input{f_37.TpX}

\input{f_33.TpX}

\end{example}
	
\begin{lem}\label{lem:f233}
{\rm (1)} $f_2$ has only unbounding invariant spin structures; 

{\rm (2)} $f_3$ has only unbounding invariant spin structures;

{\rm (3)} $f_3'$ has only bounding invariant spin structures.
\end{lem}	

\begin{proof} 
	(1) follows from \cite[Proposition 5.2 (2)]{WsWz}.  
	(2) and (3) will be proved in Lemma \ref{lem:3337} and Lemma \ref{lem:periodic}.
\end{proof}

\begin{proof}[Proof of Proposition \ref{prop:period8}]	
Since $f_2$, $f_3$ and $f'_3$ are all of period $8$ and each have two isolated fixed points, we can make periodic connected sums
of them to construct periodic maps $h_g$ of order 8 on $F_g$ for $g\ge 2$, $g\ne 4$, which are not extendable over $S^4$. The constructions are as below:

(1)	If $g\equiv 0\,({\rm mod}\,4)$ and $g\geq 8$, 
	let $h_g\in\Aut(F_g)$  be a periodic connected sum of 
	\begin{displaymath}
		f_3,f_3,\underbrace{f_2,f_2,\cdots,f_2}_{(g-6)/2}.
	\end{displaymath}
	
(2)	If $g\equiv 1\,({\rm mod}\,4)$ and $g\geq 9$, let $h_g\in\Aut(F_g)$  be a  periodic connected sum of 
\begin{displaymath}
	f_3,f_3,f_3,\underbrace{f_2,f_2,\cdots,f_2}_{(g-9)/2}.
\end{displaymath}

(3) If $g\equiv 2\,({\rm mod}\,4)$, let $h_g\in\Aut(F_g)$  be a  periodic connected sum of 
\begin{displaymath}
	\underbrace{f_2,f_2,\cdots,f_2}_{g/2}.
\end{displaymath}
	
(4) If $g\equiv 3\,({\rm mod}\,4)$, let  $h_g\in\Aut(F_g)$  be  a periodic connected sum of 
	\begin{displaymath}
		f_3,\underbrace{f_2,f_2,\cdots,f_2}_{(g-3)/2}.
	\end{displaymath}

(5) If $g=5$, let $h_g\in\Aut(F_5)$  be a periodic connected sum of $f_2$ and $f'_3$. 

In each case, the total number of $f_2$ and $f_3$ in the connected sum of $h_g$ is odd.  
By Lemma \ref{lem:f233} and Proposition \ref{prop:periodicconnectedsum}, $h_g$ is not extendable over $S^4$.
\end{proof}

\begin{example}
	A Wiman map $w_1$ on the torus is not extendable over $S^4$ thus has only an unbounding invariant spin structure. 
	So a connected sum of $w_1$ and $f_3'\in\Aut(F_3)$ is an example for $F_4$ that is not extendable.
\end{example}

\section{Invariant quadratic forms of torsion mapping classes on low genus surfaces}\label{sect:low-genus}



According to Theorem \ref{thm:extendability}, to tell whether $f\in \Aut(F_g)$ is extendable, we only need to:  choose a basis $\{a_1,a_2,\cdots,a_{2g}\}$ of $H_1(F_g;\mathbb{Z}_2)$ and calculate $f_*(a_i)$ (say, $f_*(a_i)=\sum_{j=1}^{2g}t_{ij}a_j$, $t_{ij}\in\mathbb{Z}_2$),
then solve the equation system 
\begin{equation}\label{eq:5.0}
	q(a_i)=q(f_*(a_i)),\,i=1,2,\cdots,2g
\end{equation}
for quadratic form $q$ on $H_1(F_g;\mathbb{Z}_2)$,
and finally check whether $\Arf(q)$ is $0$.
Note that using the following lemma, which can be verified directly, we can transform the equation system \eqref{eq:5.0} for $q\in\mathcal{Q}(F_g)$ into a $\mathbb{Z}_2$-linear equation system for $q(a_1),q(a_2),\cdots,q(a_{2g})\in\mathbb{Z}_2$:
$$q(a_i)=\sum_{j=1}^{2g}t_{ij}q(a_j)+\sum_{1\leq j<k\leq 2g}t_{ij}t_{ik}a_j\cdot a_k.$$

\begin{lem}\label{lem:quadratic-sum}
	For $q\in\mathcal{Q}(F_g)$ and $x_1,x_2,\cdots,x_m\in H_1(F_g;\mathbb{Z}_2)$, we have
	\begin{displaymath}
		q\left(\sum_{i=1}^{m}x_i\right)=\sum_{i=1}^{m}q(x_i)+\sum_{1\leq i<j\leq m}x_i\cdot x_j.
	\end{displaymath}
\end{lem}

Once $f\in\Aut(F_g)$ is given as a product of Dehn twists, it is convenient to compute $f_*(a_i)$ thus to decide whether $f$ is extendable over $S^4$.
The famous Dehn-Lickorish theorem tells that each mapping class on $F_g$ can be written as a product of Dehn twists.
In \cite{Hi2}, Hirose gave such presentations up to conjugacy for all torsion mapping classes on $F_g$ with $g\leq 4$. 
Our calculations in this section rely on his work.

We use $T_c$ to denote the mapping class of a right Dehn twist along a simple closed curve $c$ on $F_g$.
Writing a composition of mapping classes, we always begin from right to left, e.g., $T_{c}T_{c'}$ acts on $F_g$ as $T_{c'}$ first and then $T_{c}$. 

\input{Dehn3.TpX}

Take curves $c_1,\cdots,c_7$ and $d_1$ on $F_3$ as in Fig. \ref{fig:Dehn3}.
According to \cite[Proposition 3.1, Theorem 3.2]{Hi2}, the mapping classes
\begin{displaymath}
	T_{c_1}T_{c_2}T_{c_3}T_{c_4}T_{c_5}T_{c_6}T_{c_7}
\end{displaymath}
and
\begin{displaymath}
	T_{d_1}T_{c_3}T_{c_4}T_{c_5}T_{c_2}T_{c_3}T_{c_4}T_{c_5}T_{c_6}
\end{displaymath}
are of order 8 in $\MCG(F_3)$. 
Hirose denoted them by $f_{3,3},f_{3,7}$ respectively.
	
\begin{lem}\label{lem:3337}
(1) Each invariant quadratic form of $f_{3,3}$ has Arf invariant $0$.

(2) Each invariant quadratic form of $f_{3,7}$ has Arf invariant $1$.
\end{lem}

\begin{proof}
	Abusing the notations, for a simple closed curve $c$, we still denote by $c$ the $\mathbb{Z}_2$-homology class it represents.
	Since $\{c_1,c_2,\cdots,c_6\}$ is a basis of  $H_1(F_3;\mathbb{Z}_2)$, a quadratic form $q\in\mathcal{Q}(F_3)$ is preserved by $f\in\MCG(F_3)$ if and only if the following equation system holds:
	\begin{equation}\label{eq:5.1}
		q(c_i)=q(f(c_i)),\,i=1,2,\cdots,6.
	\end{equation}
	
	(1) Recall that a Dehn twist $T_c\in \MCG(F_3)$ acts on $H_1(F_3;\mathbb{Z}_2)$ as
	\begin{displaymath}
		T_c(c_i)=c_i+(c\cdot c_i)c,
	\end{displaymath}
	and that the mod-2 intersection number of $c_i,c_j(i,j\in\{1,2,\cdots,7\})$ is
	\begin{displaymath}
		c_i\cdot c_j=
		\begin{cases}
			1,\text{ if }|i-j|=1;\\
			0,\text{ otherwise}.
		\end{cases}
	\end{displaymath}
	We can calculate $f_{3,3}(c_i)\in H_1(F_3;\mathbb{Z}_2)(i=1,2,\cdots,6)$ as follows:
	\begin{displaymath}
		\setlength{\arraycolsep}{0.5pt}
		\begin{array}{llllllll}
			 &T_{c_7} &T_{c_6} &T_{c_5} &T_{c_4} &T_{c_3} &T_{c_2} &T_{c_1}\\
			c_1&\longrightarrow c_1&\longrightarrow c_1&\longrightarrow c_1&\longrightarrow c_1&\longrightarrow c_1&\longrightarrow c_1+c_2&\longrightarrow c_2\\
			c_2&\longrightarrow c_2&\longrightarrow c_2&\longrightarrow c_2&\longrightarrow c_2&\longrightarrow c_2+c_3&\longrightarrow c_3&\longrightarrow c_3\\
			c_3&\longrightarrow c_3&\longrightarrow c_3&\longrightarrow c_3&\longrightarrow c_3+c_4&\longrightarrow c_4&\longrightarrow c_4&\longrightarrow c_4\\
			c_4&\longrightarrow c_4&\longrightarrow c_4&\longrightarrow c_4+c_5&\longrightarrow c_5&\longrightarrow c_5&\longrightarrow c_5&\longrightarrow c_5\\
			c_5&\longrightarrow c_5&\longrightarrow c_5+c_6&\longrightarrow c_6&\longrightarrow c_6&\longrightarrow c_6&\longrightarrow c_6&\longrightarrow c_6\\
			c_6&\longrightarrow c_6+c_7&\longrightarrow c_7&\longrightarrow c_7&\longrightarrow c_7&\longrightarrow c_7&\longrightarrow c_7&\longrightarrow c_7
		\end{array}
	\end{displaymath}
	In summary,  $$f_{3,3}(c_i)=c_{i+1}\in H_1(F_3;\mathbb{Z}_2),\,i=1,2,\cdots,6.$$
	Note that
	$$c_7=c_1+c_3+c_5\in H_1(F_3;\mathbb{Z}_2)$$
	because the curves $c_1,c_3,c_5,c_7$ on $F_g$ together bound a subsurface.
	The equation system \eqref{eq:5.1} for $f_{3,3}$ now becomes
	\begin{displaymath}
		\left\{
		\begin{array}{l}
			q(c_1)=q(c_2), \\
			q(c_2)=q(c_3), \\
			q(c_3)=q(c_4), \\
			q(c_4)=q(c_5), \\
			q(c_5)=q(c_6), \\
			q(c_6)=q(c_1)+q(c_3)+q(c_5).
		\end{array}
		\right.
	\end{displaymath}
	It has two solutions:
	\begin{displaymath}
		q(c_1)=q(c_2)=q(c_3)=q(c_4)=q(c_5)=q(c_6)=0;
	\end{displaymath}
	or 
	\begin{displaymath}
		q(c_1)=q(c_2)=q(c_3)=q(c_4)=q(c_5)=q(c_6)=1.
	\end{displaymath}
	That is to say, there are exactly two $f_{3,3}$-invariant quadratic forms on $H_1(F_3;\mathbb{Z}_2)$.
	Comparing Fig. \ref{fig:Dehn3} with Fig. \ref{fig:symplectic}, we see that in $\mathbb{Z}_2$-homology,
	\begin{displaymath}
		a_i=c_{2i},\,b_i=\sum_{j=1}^{i}c_{2j-1},\,i=1,2,3.
	\end{displaymath}
	So for $q\in\mathcal{Q}(F_3)$, the Arf invariant is
	\begin{displaymath}
		\Arf(q)=\sum_{i=1}^{3}q(a_i)q(b_i)=\sum_{i=1}^{3}q(c_{2i})\sum_{j=1}^{i}q(c_{2j-1}).
	\end{displaymath}
	It follows that each $f_{3,3}$-invariant quadratic form has Arf invariant $0$.
	
	(2) For $f_{3,7}=T_{d_1}T_{c_3}T_{c_4}T_{c_5}T_{c_2}T_{c_3}T_{c_4}T_{c_5}T_{c_6}$, with similar calculations we obtain the following results in $H_1(F_3;\mathbb{Z}_2)$ (noting that $d_1=c_1+c_3$):
	\begin{displaymath}
		\begin{split}
			f_{3,7}(c_1) &=c_1+c_2+c_3,\\
			f_{3,7}(c_2) &=c_4+d_1=c_1+c_3+c_4,\\
			f_{3,7}(c_3) &=c_5,\\
			f_{3,7}(c_4) &=c_3+c_4+c_5+d_1=c_1+c_4+c_5,\\
			f_{3,7}(c_5) &=c_3+c_4+c_5+c_6+d_1=c_1+c_4+c_5+c_6,\\
			f_{3,7}(c_6) &=c_2+c_3+c_4+c_5+c_6+d_1=c_1+c_2+c_4+c_5+c_6.\\
		\end{split}
	\end{displaymath}
	Using Lemma \ref{lem:quadratic-sum}, we see that the equation system \eqref{eq:5.1} for $q\in\mathcal{Q}(F_3)$ to be $f_{3,7}$-invariant is
	\begin{displaymath}
		\left\{
		\begin{array}{l}
			q(c_1)=q(c_1)+q(c_2)+q(c_3),\\
			q(c_2)=q(c_1)+q(c_3)+q(c_4)+1,\\
			q(c_3)=q(c_5),\\
			q(c_4)=q(c_1)+q(c_4)+q(c_5)+1,\\
			q(c_5)=q(c_1)+q(c_4)+q(c_5)+q(c_6),\\
			q(c_6)=q(c_1)+q(c_2)+q(c_4)+q(c_5)+q(c_6)+1.
		\end{array}
		\right.
	\end{displaymath}
	It also has two solutions:
	\begin{displaymath}
		q(c_1)=0,\,q(c_2)=q(c_3)=q(c_4)=q(c_5)=1,\,q(c_6)=1;
	\end{displaymath}
	or
	\begin{displaymath}
		q(c_1)=1,\,q(c_2)=q(c_3)=q(c_4)=q(c_5)=0,\,q(c_6)=1.
	\end{displaymath}
	Each has Arf invariant $1$. 
\end{proof}

\begin{lem}\label{lem:periodic}
In $\MCG(F_3)$, $f_{3,3}$ is conjugate to $[f_{3}']$, and $f_{3,7}$ is conjugate to $[f_{3}]$, where $f_3,f_3'\in\Aut(F_3)$ are the maps of period $8$ in Example \ref{ex:f233}.			
\end{lem}

To prove that, we here introduce Nielsen's classification of periodic maps on oriented closed surfaces.
Suppose $f\in\Aut(F_g)$ is of period $n$.
Let $\Sigma_f$ be the set
$$\{x\in F_g\,|\, \exists\,k,\text{ s.t. } 1<k<n,\,f^k(x)=x\}.$$
Then the quotient map 
$$F_g-\Sigma_f\to (F_g-\Sigma_f)/f$$ 
is a covering of surfaces,
corresponding to an epimorphism
$$\pi_1((F_g-\Sigma_f)/f)\twoheadrightarrow\langle f\rangle.$$
Identify $\langle f\rangle$ with the addative group $\mathbb{Z}_n$ by identifying $f$ with $1$, then the epimorphism factors through as 
$$\psi_f: H_1((F_g-\Sigma_f)/f;\mathbb{Z})\twoheadrightarrow\mathbb{Z}_n.$$
Suppose the surface $(F_g-\Sigma_f)/f$ has $s$ punctures, each corresponding to an $f$-orbit of points in $\Sigma_f$. 
For each puncture $p_i(1\leq i\leq s)$, take a closed curve $\xi_i:[0,1]\to(F_g-\Sigma_f)/f$ surrounding $p_i$ with orientation induced by the surface.
Let $\tilde{\xi}_i:[0,1]\to F_g-\Sigma_f$ be a lift of $\xi_i$, then $\psi_f(\xi_i)$ is the number satisfying 
$$f^{\psi_f(\xi_i)}(\tilde{\xi}_i(0))=\tilde{\xi}_i(1).$$

\begin{thm}\label{thm:classification}\cite{Ni}
	A periodic map $f$ on an oriented closed surface $F_g$ is uniquely determined up to conjugacy by the following invariants:
	its period $n$, the number $s$ of punctures on $(F_g-\Sigma_f)/f$, and the multiset $\{\psi_f(\xi_1),\psi_f(\xi_2),\cdots,\psi_f(\xi_s)\}$. 
\end{thm}

\begin{proof}[Proof of Lemma \ref{lem:periodic}]
	The invariants of $f_{3,3}, f_{3,7}$ are given in \cite[Proposition 3.1]{Hi2}:
	\begin{displaymath}
		\begin{split}
			&\text{for}\, f_{3,3},\,n=8,\,s=3,\,\{\psi_f(\xi_1),\psi_f(\xi_2),\psi_f(\xi_3)\}=\{1,1,6\};\\
			&\text{for}\, f_{3,7},\,n=8,\,s=3,\,\{\psi_f(\xi_1),\psi_f(\xi_2),\psi_f(\xi_3)\}=\{1,2,5\}.
		\end{split}
	\end{displaymath}

	\input{f_3.TpX}
	From Fig. \ref{fig:f_37}, we see that $f_3$ has $3$ singular orbits: the center of the octagon, the centers of the squares, and all the vertices.
	Take $\tilde{\xi}_1,\tilde{\xi}_2,\tilde{\xi}_3$ as in Fig. \ref{fig:f_3}, then it is clear that the multiset invariant for $f_3$ is $\{1,2,5\}$. 
	By Theorem \ref{thm:classification}, $f_3$ is conjugate to a periodic map in the isotopy class of $f_{3,7}$.
	Similarly one can verify that $[f_3']$ is conjugate to $f_{3,3}$.
\end{proof}

\begin{lem}\cite[Theorem 3.2]{Hi2}
	Take curves $c_1,\cdots,c_9$ and $d_1,d_2,d_1',d_2',e$ on $F_4$ as in Fig. \ref{fig:Dehn4}.
	Each torsion in $\MCG(F_4)$ is conjugate to a power of one of the following 12 mapping classes:
	\begin{displaymath}
		\begin{split}
			f_{4,1} &=T_{c_1}T_{c_2}T_{c_3}T_{c_4}T_{c_5}T_{c_6}T_{c_7}T_{c_8}\,(\text{of order }18),\\   f_{4,2} &=f_{4,1}T_{c_8}\,(\text{of order }16),\\  
			f_{4,3} &=f_{4,1}T_{c_9}\,(\text{of order }10),\\   
			f_{4,4} &=f_{4,3}^5T_{c_9}T_{c_8}T_{c_7}T_{c_6}T_{c_5}T_{c_4}T_{c_3}T_{c_2}T_{c_1}\,(\text{of order }10),\\  
			f_{4,5} &=T_{d_2}T_{c_2}T_{c_3}T_{c_4}T_{c_5}T_{c_6}T_{c_7}T_{c_8}\,(\text{of order }15),\\  
			f_{4,6} &=f_{4,5}T_{c_5}T_{c_6} \,(\text{of order }12),\\  
			f_{4,7} &=T_{d_2}T_{c_4}T_{c_5}T_{c_6}T_{c_7}T_{c_2}T_{c_3}T_{c_4}T_{c_5}T_{c_6}T_{c_7}T_{c_8}\,(\text{of order }10),\\  
			f_{4,8} &=T_{e}T_{d_2}T_{c_6}T_{c_8}T_{c_7}T_{c_6}T_{c_5}T_{c_4}T_{c_3}\,(\text{of order }12),\\  
			f_{4,9}& =T_{d_2}^{-1}T_{c_6}^{-1}T_{c_7}^{-1}T_{c_8}^{-1}T_{c_9}^{-1}T_{d_1}T_{c_4}T_{c_3}T_{c_2}T_{c_1}\,(\text{of order }6),\\  
			f_{4,10} &=T_{c_9}T_{c_8}T_{d_2}^{-1}T_{c_6}^{-1}T_{c_5}^{-1}T_{c_4}^{-1}T_{d'_1}^{-1}T_{c_2}T_{c_1}\,(\text{of order }6),\\  
			f_{4,11} &=T_{c_9}^{-1}(T_{c_7}T_{d'_2}T_{c_6}T_{c_5}T_{c_4}T_{c_3}T_{c_2})^2\,(\text{of order }6),\\  
			f_{4,12} &=T_{c_7}^{-1}T_{d_2}^{-1}T_{c_6}^{-1}T_{c_7}^{-1}T_{d'_2}^{-1}T_{c_6}^{-1}T_{c_7}^{-1}T_{c_8}^{-1}T_{c_3}T_{d_1}T_{c_4}T_{c_3}T_{d'_1}T_{c_4}T_{c_3}T_{c_2}\,(\text{of order }5).
		\end{split}
	\end{displaymath}
	\input{Dehn4.TpX}
\end{lem}

With the following proposition, we finish the proof of Theorem \ref{non-extendable}.

\begin{prop}\label{prop:torsion4}
	Every torsion mapping class on $F_4$ is extendable over $S^4$.
\end{prop}

\begin{proof}
	It suffices to verify that the 12 mapping classes above all have invariant quadratic forms with Arf invariant $0$. 
	
	We list each $f_{4,i}(c_j)$ in Table \ref{table:genus4}. 
	Note that in $H_1(F_4;\mathbb{Z}_2)$ we have
	\begin{displaymath}
		\begin{split}
			d_1 &=d'_1=c_1+c_3,\\ 
			d_2 &=d'_2=c_1+c_3+c_5,\\
			e\ &=c_1+c_2+c_3+c_4+c_6,\\
			c_9 &=c_1+c_3+c_5+c_7.
		\end{split}
	\end{displaymath}

	\begin{table}[htbp]
	\caption{The actions of $f_{4,i}$'s on $H_1(F_4;\mathbb{Z}_2)$.}
	\label{table:genus4}
	\centering
	\footnotesize
	\begin{tabularx}{370pt}{>{\centering\arraybackslash}X|>{\centering\arraybackslash}m{31pt}>{\centering\arraybackslash}m{31pt}>{\centering\arraybackslash}m{31pt}>{\centering\arraybackslash}m{31pt}>{\centering\arraybackslash}m{31pt}>{\centering\arraybackslash}m{31pt}>{\centering\arraybackslash}m{31pt}>{\centering\arraybackslash}m{31pt}}
		\toprule
		$c_i$ & $c_1$ & $c_2$ & $c_3$ & $c_4$ & $c_5$ & $c_6$ & $c_7$ & $c_8$ \\
		\hline
		$f_{4,1}(c_i)$ & $c_2$ & $c_3$ & $c_4$ & $c_5$ & $c_6$ & $c_7$ & $c_8$ & $\sum\limits_{j=1}^{8}c_j$ \\
		\hline
		$f_{4,2}(c_i)$ & $c_2$ & $c_3$ & $c_4$ & $c_5$ & $c_6$ & $c_7$ & $\sum\limits_{j=1}^{7}c_j$ & $\sum\limits_{j=1}^{8}c_j$ \\
		\hline
		$f_{4,3}(c_i)$ & $c_2$ & $c_3$ & $c_4$ & $c_5$ & $c_6$ & $c_7$ & $c_8$ & $c_1+c_3+c_5+c_7$ \\
		\hline
		$f_{4,4}(c_i)$ & $c_5$ & $c_6$ & $c_7$ & $c_8$ & $c_1+c_3+c_5+c_7$ & $c_2+c_4+c_6+c_8$ & $c_1$ & $c_2$ \\
		\hline
		$f_{4,5}(c_i)$ & $c_1+c_2$ & $c_3$ & $c_4$ & $c_5$ & $c_1+c_3+c_5+c_6$ & $c_7$ & $c_8$ & $c_1+c_2+c_4+c_6+c_7+c_8$ \\
		\hline
		$f_{4,6}(c_i)$ & $c_1+c_2$ & $c_3$ & $c_4$ & $c_1+c_3+c_6$ & $c_7$ & $c_1+c_3+c_5+c_6+c_7$ & $c_1+c_3+\sum\limits_{j=5}^{8}c_j$ & $c_1+c_2+c_4+c_6+c_7+c_8$ \\
		\hline
		$f_{4,7}(c_i)$ & $c_1+c_2$ & $c_3+c_4$ & $c_5$ & $c_1+c_3+c_5+c_6$ & $c_7$ & $c_1+c_3+c_4+c_6+c_7$ & $c_1+c_3+c_4+c_6+c_7+c_8$ & $c_1+c_2+c_4+c_6+c_7+c_8$ \\
		\hline
		$f_{4,8}(c_i)$ & $c_2+c_3+c_4+c_6$ & $c_3+c_7+c_8$ & $c_2+c_3+c_7+c_8$ & $c_3$ & $c_1+c_2+c_3+c_6$ & $c_2+c_4$ & $c_2+c_4+c_5$ & $c_1+c_3+c_5+c_6+c_7$ \\
		\hline
		$f_{4,9}(c_i)$ & $c_2+c_4$ & $c_1$ & $c_2$ & $c_3$ & $c_4+c_6$ & $c_7$ & $c_8$ & $c_1+c_3+c_5+c_7$ \\
		\hline
		$f_{4,10}(c_i)$ & $c_1+c_2$ & $c_1$ & $c_1+c_2+c_4+c_6$ & $c_1+c_3$ & $c_4$ & $c_5$ & $c_6+c_8$ & $c_1+c_3+c_5+c_7+c_8$ \\
		\hline
		$f_{4,11}(c_i)$ & $c_2+c_4+c_7$ & $c_6$ & $c_1+c_2+c_4+c_6+c_7$ & $c_2$ & $c_3$ & $c_4$ & $c_5+c_6+c_7$ & $c_6+c_8$ \\
		\hline
		$f_{4,12}(c_i)$ & $\sum\limits_{j=1}^{4}c_j$ & $c_2+c_3+c_4$ & $c_1+c_2+c_4$ & $c_1+c_3+c_4$ & $c_3+c_7$ & $c_1+c_3+c_5+c_6$ & $c_1+c_3+\sum\limits_{j=5}^{8}c_j$ & $c_6+c_7+c_8$ \\
		\bottomrule
	\end{tabularx}
	\end{table}
	
	It can be verified that
	$f_{4,1},f_{4,2},f_{4,5},f_{4,6},f_{4,7},f_{4,9}$ preserve the quadratic form $q_1\in\mathcal{Q}(F_4)$ with
	\begin{displaymath}
		q_1(c_i)=1(1\leq i\leq 8);
	\end{displaymath}
	$f_{4,3},f_{4,4},f_{4,11}$ preserve $q_2$ with
	\begin{displaymath}
		q_2(c_i)=0(1\leq i\leq 8);
	\end{displaymath}
	$f_{4,8}$ preserves $q_3$ with
	\begin{displaymath}
		q_3(c_1)=q_3(c_2)=\cdots=q_3(c_5)=q_3(c_8)=1,\,q_3(c_6)=q_3(c_7)=0;
	\end{displaymath}
	$f_{4,10}$ preserves $q_4$ with
	\begin{displaymath}
		q_4(c_1)=q_4(c_2)=q_4(c_3)=q_4(c_7)=q_4(c_8)=1,\,q_4(c_4)=q_4(c_5)=q_4(c_6)=0;
	\end{displaymath}
	and $f_{4,12}$ preserves $q_5$ with
	\begin{displaymath}
		q_5(c_1)=q_5(c_5)=0,\,q_5(c_2)=q_5(c_3)=q_5(c_4)=q_5(c_6)=q_5(c_7)=q_5(c_8)=1.
	\end{displaymath}
	Again, the Arf invariant of $q\in\mathcal{Q}(F_4)$ can be computed by
	\begin{displaymath}
		{\rm Arf}(q)=\sum_{i=1}^{4}q(c_{2i})\sum_{j=1}^{i}q(c_{2j-1}).
	\end{displaymath}
	So each ${\rm Arf}(q_i)(i=1,2,3,4,5)$ turns out to be $0$.
\end{proof}

\begin{rem}
	According to \cite[Theorem 1.2, Theorem 1.3]{NWW} and the data given in \cite[Proposition 3.1]{Hi2}, $f_{4,i}\in \Aut(F_4)$ is extendable over the $3$-sphere if and only if $i=4,9,11,12$.
	And $f_3'\in\Aut(F_3)$ is not extendable over the $3$-sphere while it is extendable over $S^4$.
\end{rem}

\section{Constructing extendable embeddings $F_g\hookrightarrow S^4$}\label{sect:embedding}

Now we illustrate how to construct a (trivial) embedding $e: F_g\hookrightarrow S^4$ for an automorphism $f$ of $F_g$, which has an invariant quadratic form $q$ with Arf invariant $0$,  so that $f$ is extendable with respect to $e$.

Take a basis $\{a_1,b_1,\cdots,a_g,b_g\}$ of $H_1(F_g;\mathbb{Z}_2)$ as in Fig. \ref{fig:symplectic}. 
Divide $\{1,2,\cdots,g\}$ into $4$ subsets
$$Q_{r,s}=\{i\,|\,q(a_i)=r,q(b_i)=s\},\,r,s\in\mathbb{Z}_2.$$
As ${\rm Arf}(q)=\sum_{i=1}^{g}q(a_i)q(b_i)=0$, the cardinality of $Q_{1,1}$ is even, and we may write
$$Q_{1,1}=\{i_1,j_1,i_2,j_2,\cdots,i_t,j_t\}$$
with $1\leq i_1<j_1<i_2<j_2<\cdots<i_t<j_t\leq g$.
For each $k=1,2,\cdots,t$, take two simple closed curves $c_k$ and $d_k$ on $F_g$ representing $a_{i_k}+a_{j_k}$ and $b_{i_k}+b_{j_k}$ respectively (Fig. \ref{fig:curves}).
Let $h$ be a map representing the product
$$\prod_{i\in Q_{0,1}}T_{a_i}\prod_{j\in Q_{1,0}}T_{b_j}\prod_{k=1}^{t}T_{c_k}T_{d_k}$$
of Dehn twists.
If we view Fig. \ref{fig:symplectic} as an obvious embedding of $F_g$ in $S^3$ and view $S^3$ as the equator of $S^4$, then we have a trivial embedding, denoted by $e_0$.

\input{symplectic1.TpX}

\begin{claim}\label{claim}
	The map $f$ is extendable over $S^4$ with respect to the embedding $e_0\circ h$.
\end{claim}

Moreover, for an automorphism of $F_{g,k}$ which satisfies the condition in Proposition \ref{prop:finitetype}, we can extends it to an automorphism of a closed surface as in the proof, and then construct an embedding in $S^4$ for its extension.

To prove Claim \ref{claim}, we should first introduce the induced Rohlin form $q_e$ for an embedding $e:F_g\hookrightarrow S^4$ as follows.
Suppose $P\subset S^4$ is an embedded surface with $\partial P\subset e(F_g)$ and $P-\partial P$ intersects $e(F_g)$ transversely.
Perturb $P$ to obtain $P'$ such that $P'$ is transverse to both $P,e(F_g)$, and $\partial P'\subset e(F_g)$ is parallel to $\partial P$.
Then the value of $q_e$ on $[e^{-1}(\partial P)]\in H_1(F_g;\mathbb{Z}_2)$ equals $|P\cap P'|\,({\rm mod}\,2)$.
It defines a quadratic form $q_e:H_1(F_g;\mathbb{Z}_2)\to\mathbb{Z}_2$ \cite{Roh}.

\begin{thm}\cite[Theorem 1.2]{Hi1}\label{thm:Hirose}
	Suppose $e:F_g\hookrightarrow S^4$ is a trivial embedding. 
	An automorphism $f$ of $F_g$ is extendable over $S^4$ with respect to $e$ if and only if $f^*q_e=q_e$.
\end{thm}

\begin{proof}[Proof of Claim \ref{claim}]
	In $H_1(F_g;\mathbb{Z}_2)$ we have
	\begin{displaymath}
		\begin{split}
			T_{c_k}T_{d_k}(a_{i_k}) & =a_{i_k}+b_{i_k}+b_{j_k},\\ T_{c_k}T_{d_k}(b_{i_k}) & =a_{i_k}+b_{i_k}+a_{j_k},\\ T_{c_k}T_{d_k}(a_{j_k}) & =a_{j_k}+b_{i_k}+b_{j_k},\\ T_{c_k}T_{d_k}(b_{j_k}) & =a_{i_k}+a_{j_k}+b_{j_k},\\
			T_{a_j}(b_j) & = a_j+b_j,\\
			T_{b_i}(a_i) & = a_i+b_i,
		\end{split}
	\end{displaymath}
	while the other generators are preserved by these Dehn twists respectively.
	Let $q_{e_0}$ be the induced Rohlin form of $e_0$. 
	For $i,j=1,2,\cdots,g$, we see that $a_i,b_j$ and their parallel copies bound disjoint disks in $S^3$ (Fig. \ref{fig:disks}),
	\input{symplectic2.TpX}
	so according to the definition, $q_{e_0}(a_i)=q_{e_0}(b_j)=0$.
	Therefore,
	\begin{displaymath}
		\begin{split}
			h^*q_{e_0}(a_i) &=q_{e_0}(h(a_i))=\left\{
			\begin{array}{ll}
				q_{e_0}(a_i)=0 & \text{ if } i\in Q_{0,0}\cup Q_{0,1}\\
				q_{e_0}(a_i+b_i)=1 & \text{ if } i\in Q_{1,0}\\
				q_{e_0}(a_{i_k}+b_{i_k}+b_{j_k})=1 & \text{ if } i=i_k\in Q_{1,1}\\
				q_{e_0}(a_{j_k}+b_{i_k}+b_{j_k})=1 & \text{ if } i=j_k\in Q_{1,1}
			\end{array}
			\right.\\
			& =q(a_i),
		\end{split}
	\end{displaymath}
	and similarly $h^*q_{e_0}(b_i)=q(b_i)$ for each $i=1,2,\cdots,g$. 
	Thus $h^*q_{e_0}=q$.

	Denote $e=e_0\circ h$.
	According to the definition of induced Rohlin form, $q_{e}(x)$ equals $q_{e_0}(h(x))$ for any $x\in H_1(F_g;\mathbb{Z}_2)$.
	As $q_{e_0}(h(x))=h^*q_{e_0}(x)=q(x)$, we have $q_e=q$, thus $f^*q_e=q_e$. 
	It follows from Theorem \ref{thm:Hirose} that $f$ is extendable over $S^4$ with respect to $e$.
\end{proof}

\input{stable.TpX}

\begin{example} 
	Below we work on a concrete example which provides intuition of  both Theorem \ref{stable} and Theorem \ref{puncture}.

	The map $f_2\in\Aut(F_2)$ in Example \ref{ex:f233} has only unbounding invariant spin structures thus is not extendable over $S^4$. 
	According to Theorem \ref{stable}, it becomes extendable after a connected sum with the identity $I_T$ on the torus $T$.
	See Fig. \ref{fig:stable}. 
	View $F_3$ as a connected sum of $T$ and $F_2$ along the blue curve in the middle.
	Take a baisi $\{a_1,b_1,a_2,b_2,a_3,b_3\}$ of $H_1(F_3;\mathbb{Z}_2)$ as in Fig. \ref{fig:symplectic}, then in $H_1(F_3;\mathbb{Z}_2)$ we have
	$$a=b_2,\,b=a_2+a_3+b_3,\,c=a_2+b_3,\,d=a_2+b_2.$$
	Note that $\{a,b,c,d\}$ is a basis of $H_1(F_2;\mathbb{Z}_2)$ and is an $f_2$-orbit, so a quadratic form on $F_3$ is invariant under $[I_T]\#[f_2]$ if and only if it is constant on $a,b,c,d$. 
	Now fix an invariant quadratic form $q$ with $q(a_1)=q(b_1)=q(a)=q(b)=q(c)=q(d)=1$.
	Then
	\begin{displaymath}
		\begin{split}
			q(a_2)=q(a+d)=1,\, & q(b_2)=q(a)=1,\\
			q(a_3)=q(b+c)=1,\,& q(b_3)=q(a+c+d)=0.
		\end{split}
	\end{displaymath}
	Take the trivial embedding $e_0:F_3\hookrightarrow S^3\subset S^4$ as Fig. \ref{fig:symplectic} presents and let $h=T_{b_3}T_{c_1}T_{d_1}$, where $b_3,c_1,d_1$ are the black curves in the middle of Fig. \ref{fig:stable}.	
	Then by Claim \ref{claim}, $[I_T]\#[f_2]$ is extendable over $S^4$ with respect to the embedding $e_0\circ h$ (the bottom of Fig. \ref{fig:stable}).
	
	Related to Theorem \ref{puncture}, the blue curve, which is a figure-eight knot in $S^3$, bounds an embedded $F_{2,1}$ on the right side.
	The restriction of $f_2$ on $F_{2,1}$ (obtained from $F_2$ by removing an open neighborhood of a fixed point of $f_2$) is also extendable over $S^4$ with respect to this embedding.
\end{example}

\bibliographystyle{amsalpha}

\end{document}